\documentclass[11pt]{amsart}
\usepackage{amssymb,amscd,amsmath,enumerate}

\newtheorem{theorem}{Theorem}[section]
\newtheorem{lemma}[theorem]{Lemma}
\newtheorem{corollary}[theorem]{Corollary}
\newtheorem{proposition}[theorem]{Proposition}

\theoremstyle{definition}
\newtheorem{example}[theorem]{Example}
\newtheorem{remark}[theorem]{Remark}
\numberwithin{equation}{theorem}

\def\nil{\mathrm{nil}}
\def\GR{{+\mathrm{GR}}}
\def\sep{{+\mathrm{sep}}}

\def\ge{\geqslant}

\def\bar{\overline}
\def\del{\partial}

\def\to{\longrightarrow}
\def\ff{\operatorname{frac}}
\def\trace{\operatorname{tr}}
\def\Ass{\operatorname{Ass}}
\def\Ext{\operatorname{Ext}}

\def\GL{\operatorname{GL}}
\def\PSL{\operatorname{PSL}}
\def\Gal{\operatorname{Gal}}

\def\bsx{{\boldsymbol{x}}}

\def\fraka{\mathfrak{a}}
\def\frakp{\mathfrak{p}}
\def\frakq{\mathfrak{q}}
\def\frakm{\mathfrak{m}}
\def\frakM{\mathfrak{M}}

\def\FF{\mathbb{F}}
\def\NN{\mathbb{N}}
\def\QQ{\mathbb{Q}}
\def\ZZ{\mathbb{Z}}

\begin{document}
\title[Maps on local cohomology]{Galois extensions, plus closure, and maps on local cohomology}

\author{Akiyoshi Sannai}
\email{sannai@ms.u-tokyo.ac.jp}
\address{Graduate School of Mathematical Sciences, The University of Tokyo, 3-8-1 Komaba, Meguro, Tokyo, 153-8914 Japan}

\author{Anurag K. Singh}
\email{singh@math.utah.edu}
\address{Department of Mathematics, University of Utah, 155 S. 1400 E., Salt Lake City, UT~84112, USA}

\thanks{The first author was supported by the Japan Society for Promotion of Science (JSPS), and the second author by NSF grant DMS~0856044.}

\begin{abstract}
Given a local domain $(R,\frakm)$ of prime characteristic that is a homomorphic image of a Gorenstein ring, Huneke and Lyubeznik proved that there exists a module-finite extension domain $S$ such that the induced map on local cohomology modules $H^i_\frakm(R)\to H^i_\frakm(S)$ is zero for each $i<\dim R$. We prove that the extension $S$ may be chosen to be generically Galois, and analyze the Galois groups that arise.
\end{abstract}
\maketitle

\section{Introduction}

Let $R$ be a commutative Noetherian integral domain. We use $R^+$ to denote the integral closure of $R$ in an algebraic closure of its fraction field. Hochster and Huneke proved the following:

\begin{theorem}\cite[Theorem~1.1]{HHbig}
\label{theorem:HH}
If $R$ is an excellent local domain of prime characteristic, then each system of parameters for $R$ is a regular sequence on $R^+$, i.e., $R^+$ is a balanced big Cohen-Macaulay algebra for $R$.
\end{theorem}

It follows that for a ring $R$ as above, and $i<\dim R$, the local cohomology module $H^i_\frakm(R^+)$ is zero. Hence, given an element $[\eta]$ of $H^i_\frakm(R)$, there exists a module-finite extension domain $S$ such that $[\eta]$ maps to $0$ under the induced map $H^i_\frakm(R)\to H^i_\frakm(S)$. This was strengthened by Huneke and Lyubeznik, albeit under mildly different hypotheses:

\begin{theorem}\cite[Theorem~2.1]{HL}
\label{thm:HL}
Let $(R,\frakm)$ be a local domain of prime characteristic that is a homomorphic image of a Gorenstein ring. Then there exists a module-finite extension domain~$S$ such that the induced map
\[
H^i_\frakm(R)\to H^i_\frakm(S)
\]
is zero for each $i<\dim R$.
\end{theorem}

By a \emph{generically Galois extension} of a domain $R$, we mean an extension domain $S$ that is integral over $R$, such that the extension of fraction fields is Galois; $\Gal(S/R)$ will denote the Galois group of the corresponding extension of fraction fields. We prove the following:

\begin{theorem}\label{thm:main:cohomology}
Let $R$ be a domain of prime characteristic.
\begin{enumerate}[\,\rm(1)]
\item Let $\fraka$ be an ideal of $R$ and $[\eta]$ an element of ${H^i_\fraka(R)}_\nil$ (see~Section~\ref{subsec:nilpotent}). Then there exists a module-finite generically Galois extension $S$, with $\Gal(S/R)$ a solvable group, such that $[\eta]$ maps to $0$ under the induced map $H^i_\fraka(R)\to H^i_\fraka(S)$.

\item Suppose $(R,\frakm)$ is a homomorphic image of a Gorenstein ring. Then there exists a module-finite generically Galois extension $S$ such that the induced map $H^i_\frakm(R)\to H^i_\frakm(S)$ is zero for each $i<\dim R$.
\end{enumerate}
\end{theorem}

Set $R^\sep$ to be the $R$-algebra generated by the elements of $R^+$ that are separable over $\ff(R)$. Under the hypotheses of Theorem~\ref{thm:main:cohomology}\,(2), $R^\sep$ is a separable balanced big Cohen-Macaulay $R$-algebra; see Corollary~\ref{corollary:sepCM}. In contrast, the algebra $R^\infty$, i.e., the purely inseparable part of $R^+$, is not a Cohen-Macaulay $R$-algebra in general: take~$R$ to be an $F$-pure domain that is not Cohen-Macaulay; see \cite[page~77]{HHbig}.

For an $\NN$-graded domain $R$ of prime characteristic, Hochster and Huneke proved the existence of a $\QQ$-graded Cohen-Macaulay $R$-algebra $R^\GR$, see Theorem~\ref{thm:hh:graded}. In view of this and the preceding paragraph, it is natural to ask whether there exists a $\QQ$-graded separable Cohen-Macaulay $R$-algebra; in Example~\ref{example:grsep} we show that the answer is negative. 

In Example~\ref{example:nfgcm} we construct an $\NN$-graded domain of prime characteristic for which no module-finite $\QQ$-graded extension domain is Cohen-Macaulay. 

We also prove the following results for closure operations; the relevant definitions may be found in Section~\ref{subsec:closure}.

\begin{theorem}
\label{thm:main:ideal}
Let $R$ be an integral domain of prime characteristic, and let $\fraka$ be an ideal of $R$.
\begin{enumerate}[\,\rm(1)]
\item Given $z\in\fraka^F$, there exists a module-finite generically Galois extension $S$, with $\Gal(S/R)$ a solvable group, such that $z\in\fraka S$.

\item Given $z\in\fraka^+$, there exists a module-finite generically Galois extension $S$ such that $z\in\fraka S$.
\end{enumerate}
\end{theorem}

In Example~\ref{example:plus} we present a domain $R$ of prime characteristic where $z\in\fraka^+$ for an element~$z$ and ideal $\fraka$, and conjecture that $z\notin\fraka S$ for each module-finite generically Galois extension $S$ with $\Gal(S/R)$ a solvable group. Similarly, in Example~\ref{example:lc} we present a $3$-dimensional ring $R$ where we conjecture that $H^2_\frakm(R)\to H^2_\frakm(S)$ is nonzero for each module-finite generically Galois extension $S$ with $\Gal(S/R)$ a solvable group.

\begin{remark}
The assertion of Theorem~\ref{thm:HL} does not hold for rings of characteristic zero: Let $(R,\frakm)$ be a normal domain of characteristic zero, and $S$ a module-finite extension domain. Then the field trace map $\trace\colon\ff(S)\to\ff(R)$ provides an $R$-linear splitting of $R\subseteq S$, namely
\[
\frac{1}{[\ff(S):\ff(R)]}\trace\colon S\to R\,.
\]
It follows that the induced maps on local cohomology $H^i_\frakm(R)\to H^i_\frakm(S)$ are $R$-split. A variation is explored in \cite{RSS}, where the authors investigate whether the image of $H^i_\frakm(R)$ in $H^i_\frakm(R^+)$ is killed by elements of $R^+$ having arbitrarily small positive valuation. This is motivated by Heitmann's proof of the direct summand conjecture for rings $(R,\frakm)$ of dimension $3$ and mixed characteristic $p>0$, \cite{Heitmann}, which involves showing that the image of
\[
H^2_\frakm(R)\to H^2_\frakm(R^+)
\]
is killed by $p^{1/n}$ for each positive integer $n$.
\end{remark}

Throughout this paper, a \emph{local ring} refers to a commutative Noetherian ring with a unique maximal ideal. Standard notions from commutative algebra that are used here may be found in \cite{BH}; for more on local cohomology, consult \cite{ILL}. For the original proof of the existence of big Cohen-Macaulay modules for equicharacteristic local rings, see \cite{Hochster}.

\section{Preliminary Remarks}

\subsection{Closure operations}
\label{subsec:closure}
Let $R$ be an integral domain. The \emph{plus closure} of an ideal~$\fraka$ is the ideal $\fraka^+=\fraka R^+\cap R$.

When $R$ is a domain of prime characteristic $p>0$, we set
\[
R^\infty=\bigcup_{e\ge0}R^{1/p^e}\,,
\]
which is a subring of $R^+$. The \emph{Frobenius closure} of an ideal $\fraka$ is the ideal $\fraka^F=\fraka R^\infty\cap R$. Alternatively, set
\[
\fraka^{[p^e]}=\big(a^{p^e}\mid a\in\fraka\big)\,.
\]
Then $\fraka^F=(r\in R\mid r^{p^e}\in\fraka^{[p^e]}\text{ for some }e\in\NN)$.

\subsection{Solvable extensions}
\label{subsec:solvable}
A finite separable field extension $L/K$ is \emph{solvable} if $\Gal(M/K)$ is a solvable group for some Galois extension $M$ of~$K$ containing~$L$. Solvable extensions form a \emph{distinguished class}, i.e.,
\begin{enumerate}[\,\rm(1)]
\item for finite extensions $K\subseteq L\subseteq M$, the extension $M/K$ is solvable if and only if each of $M/L$ and $L/K$ are solvable;

\item for finite extensions $L/K$ and $M/K$ contained in a common field, if~$L/K$ is solvable, then so is the extension $LM/M$.
\end{enumerate}

A finite separable extension $L/K$ of fields of characteristic $p>0$ is solvable precisely if it is obtained by successively adjoining
\begin{enumerate}[\,\rm(1)]
\item roots of unity;

\item roots of polynomials $T^n-a$ for $n$ coprime to $p$;

\item roots of \emph{Artin-Schreier polynomials}, $T^p-T-a$.
\end{enumerate}

\subsection{Frobenius-nilpotent submodules}
\label{subsec:nilpotent}
Let $R$ be a ring of prime characteristic $p$. A \emph{Frobenius action} on an $R$-module $M$ is an additive map $F\colon M\to M$ with $F(rm)=r^pF(m)$ for each~$r\in R$ and $m\in M$. In this case, $\ker F$ is a submodule of $M$, and we have an ascending sequence
\[
\ker F\subseteq\ker F^2 \subseteq\ker F^3\subseteq\dots\,.
\]
The union of these is the \emph{$F$-nilpotent} submodule of $M$, denoted $M_\nil$. If $R$ is local and $M$ is Artinian, then there exists a positive integer $e$ such that $F^e(M_{\nil})=0$; see \cite[Proposition~4.4]{Ly-F} or \cite[Theorem~1.12]{HS}.

\section{Proofs}

We record two elementary results that will be used later:

\begin{lemma}
\label{lemma:solvable}
Let $K$ be a field of characteristic $p>0$. Let $a$ and $b$ elements of $K$ where $a$ is nonzero. Then the Galois group of the polynomial
\[
T^p+aT-b
\]
is a solvable group.
\end{lemma}

\begin{proof}
Form an extension of $K$ by adjoining a primitive $p-1$ root of unity and an element $c$ that is a root of $T^{p-1}-a$. The polynomial $T^p+aT-b$ has the same roots as
\[
\Big(\frac{T}{c}\Big)^p-\Big(\frac{T}{c}\Big)-\frac{b}{{c\,}^p}\,,
\]
which is an Artin-Schreier polynomial in $T/c$.
\end{proof}

\begin{lemma}
\label{lemma:unlocalize}
Let $R$ be a domain, and $\frakp$ a prime ideal. Given a domain $S$ that is a module-finite extension of $R_\frakp$, there exists a domain $T$, module-finite over $R$, with $T_\frakp=S$.
\end{lemma}

\begin{proof}
Given $s_i\in S$, there exists $r_i\in R\smallsetminus\frakp$ such that $r_is_i$ is integral over $R$. If $s_1,\dots,s_n$ are generators for $S$ as an $R$-module, set $T=R[r_1s_1,\dots,r_ns_n]$.
\end{proof}

\begin{proof}[Proof of Theorem~\ref{thm:main:cohomology}]
Since solvable extensions form a distinguished class, (1) reduces by induction to the case where $F([\eta])=0$. Compute $H^i_\fraka(R)$ using a \v Cech complex $C^\bullet(\bsx;R)$, where $\bsx=x_0,\dots,x_n$ are nonzero elements generating the ideal $\fraka$; recall that $C^\bullet(\bsx;R)$ is the complex
\[
0\to R\to\bigoplus_{i=0}^nR_{x_i}\to\bigoplus_{i<j}R_{x_ix_j}\to\cdots\to R_{x_0\cdots x_n}\to 0\,.
\]
Consider a cycle $\eta$ in $C^i(\bsx;R)$ that maps to $[\eta]$ in $H^i_\fraka(R)$. Since $F([\eta])=0$, the cycle $F(\eta)$ is a boundary, i.e., $F(\eta)=\del(\alpha)$ for some $\alpha\in C^{i-1}(\bsx;R)$.

Let $\mu_1,\dots,\mu_m$ be the square-free monomials of degree $i-2$ in the elements $x_1,\dots,x_n$, and regard $C^{i-1}(\bsx;R)=C^{i-1}(x_0,\dots,x_n;R)$ as
\[
R_{x_0\mu_1}\oplus\cdots\oplus R_{x_0\mu_m}\oplus C^{i-1}(x_1,\dots,x_n;R)\,.
\]
There exist a power $q$ of the characteristic $p$ of $R$, and elements $b_1,\dots,b_m$ in $R$, such that $\alpha$ can be written in the above direct sum as
\[
\alpha=\left(\frac{b_1}{(x_0\mu_1)^q},\ \dots,\ \frac{b_m}{(x_0\mu_m)^q},\ *,\ \dots,\ *\right).
\]
Consider the polynomials
\[
T^p+x_0^qT-b_i\qquad\text{ for }i=1,\dots,m\,,
\]
and let $L$ be a finite extension field where these have roots $t_1,\dots,t_m$ respectively. By Lemma~\ref{lemma:solvable}, we may assume $L$ is Galois over $\ff(R)$ with the Galois group being solvable. Let $S$ be a module-finite extension of $R$ that contains $t_1,\dots,t_m$, and has $L$ as its fraction field; if $R$ is excellent, we may take $S$ to be the integral closure of $R$ in $L$.

In the module $C^{i-1}(\bsx;S)$ one then has
\[
\alpha=\left(\frac{t_1^p+x_0^qt_1}{(x_0\mu_1)^q},\ \dots,\ \frac{t_m^p+x_0^qt_m}{(x_0\mu_m)^q},\ *,\ \dots,\ *\right)=F(\beta)+\gamma\,,
\]
where
\[
\beta=\left(\frac{t_1}{(x_0\mu_1)^{q/p}},\ \dots,\ \frac{t_m}{(x_0\mu_m)^{q/p}},\ 0\ ,\ \dots,0\right)
\]
and
\[
\gamma=\left(\frac{t_1}{\mu_1^q},\ \dots,\ \frac{t_m}{\mu_m^q},\ *\ ,\dots,\ *\right)
\]
are elements of
\[
C^{i-1}(\bsx;S)=S_{x_0\mu_1}\oplus\cdots\oplus S_{x_0\mu_m}\oplus C^{i-1}(x_1,\dots,x_n;S)\,.
\]
Since $F(\eta)=\del(F(\beta)+\gamma)$, we have
\[
F\big(\eta-\del(\beta)\big)=\del(\gamma)\,.
\]
But $[\eta]=[\eta-\del(\beta)]$ in $H^i_\fraka(S)$, so after replacing $\eta$ we may assume that
\[
F(\eta)=\del(\gamma)\,.
\]
Next, note that $\gamma$ is an element of $C^{i-1}(1,x_1,\dots,x_n;S)$, viewed as a submodule of $C^{i-1}(\bsx;S)$. There exits $\zeta$ in $C^{i-2}(1,x_1,\dots,x_n;S)$ such that
\[
\del(\zeta)=\left(\frac{t_1}{\mu_1^q},\ \dots,\ \frac{t_m}{\mu_m^q},\ *\ ,\dots,\ *\right).
\]
Since
\[
F(\eta)=\del\big(\gamma-\del(\zeta)\big)\,,
\]
after replacing $\gamma$ we may assume that the first $m$ coordinate entries of $\gamma$ are $0$, i.e., that
\[
\gamma=\left(0,\ \dots,\ 0,\ \frac{c_1}{\lambda_1^Q}\ ,\dots,\ \frac{c_l}{\lambda_l^Q}\right),
\]
where $Q$ is a power of $p$, the $c_i$ belong to $S$, and $\lambda_1,\dots,\lambda_l$ are the square-free monomials of degree $i-1$ in $x_1,\dots,x_n$.

The coordinate entries of $\del(\gamma)$ include each $c_i/\lambda_i^Q$. Since $\del(\gamma)=F(\eta)$, each $c_i/\lambda_i^Q$ is a $p$-th power in $\ff(S)$; it follows that each $c_i$ has a $p$-th root in $\ff(S)$. After enlarging $S$ by adjoining each $c_i^{1/p}$, we see that $\gamma=F(\xi)$ for an element $\xi$ of $C^{i-1}(\bsx;S)$. But then
\[
F(\eta)=\del\big(F(\xi)\big)=F\big(\del(\xi)\big)\,.
\]
Since the Frobenius action on $C^i(\bsx;S)$ is injective, we have $\eta=\del(\xi)$, which proves (1).

For (2), it suffices to construct a module-finite generically separable extension $S$ such that $H^i_\frakm(R)\to H^i_\frakm(S)$ is zero for $i<\dim R$; to obtain a generically Galois extension, enlarge $S$ to a module-finite extension whose fraction field is the Galois closure of $\ff(S)$ over $\ff(R)$.

We use induction on $d=\dim R$, as in \cite{HL}. If $d=0$, there is nothing to be proved; if $d=1$, the inductive hypothesis is again trivially satisfied since $H^0_\frakm(R)=0$. Fix $i<\dim R$. Let $(A,\frakM)$ be a Gorenstein local ring that has $R$ as a homomorphic image, and set
\[
M=\Ext_A^{\dim A-i}(R,A)\,.
\]
Let $\frakp_1,\dots,\frakp_s$ be the elements of the set $\Ass_A M\smallsetminus\{\frakM\}$. 

Let $\frakq$ be a prime ideal of $R$ that is not maximal. Since $R$ is catenary, one has
\[
\dim R=\dim R_\frakq+\dim R/\frakq\,.
\]
Thus, the condition $i<\dim R$ may be rewritten as
\[
i-\dim R/\frakq<\dim R_\frakq\,.
\]
Using the inductive hypothesis and Lemma~\ref{lemma:unlocalize}, there exists a module-finite extension $R'$ of $R$ such that $\ff(R')$ is a separable field extension of $\ff(R_\frakq)=\ff(R)$, and the induced map
\begin{equation}
\label{eq:map}
H^{i-\dim R/\frakq}_{\frakq R_\frakq}(R_\frakq)\to H^{i-\dim R/\frakq}_{\frakq R_\frakq}(R'_\frakq)
\end{equation}
is zero. Taking the compositum of finitely many such separable extensions inside a fixed algebraic closure of $\ff(R)$, there exists a module-finite generically separable extension $R'$ of $R$ such that the map~\eqref{eq:map} is zero when $\frakq$ is any of the primes $\frakp_1R,\dots,\frakp_sR$. We claim that the image of the induced map $H^i_\frakm(R)\to H^i_\frakm(R')$ has finite length.

Using local duality over $A$, it suffices to show that
\[
M'=\Ext_A^{\dim A-i}(R',A)\to\Ext_A^{\dim A-i}(R,A)=M
\]
has finite length. This, in turn, would follow if
\[
M'_\frakp=\Ext_{A_\frakp}^{\dim A-i}(R'_\frakp,A_\frakp)\to\Ext_{A_\frakp}^{\dim A-i}(R_\frakp,A_\frakp)=M_\frakp
\]
is zero for each prime ideal $\frakp$ in $\Ass_AM\smallsetminus\{\frakM\}$. Using local duality over $A_\frakp$, it suffices to verify the vanishing of
\[
H^{\dim A_\frakp-\dim A+i}_{\frakp R_\frakp}(R_\frakp)\to H^{\dim A_\frakp-\dim A+i}_{\frakp R_\frakp}(R'_\frakp)
\]
for each $\frakp$ in $\Ass_AM\smallsetminus\{\frakM\}$. This, however, follows from our choice of $R'$ since
\[
\dim A_\frakp-\dim A+i\ =\ i-\dim A/\frakp\ =\ i-\dim R/\frakp R\,.
\]

What we have arrived at thus far is a module-finite generically separable extension $R'$ of $R$ such that the image of $H^i_\frakm(R)\to H^i_\frakm(R')$ has finite length; in particular, this image is finitely generated. Working with one generator at a time and taking the compositum of extensions, given $[\eta]$ in $H^i_\frakm(R')$, it suffices to construct a module-finite generically separable extension $S$ of $R'$ such that $[\eta]$ maps to $0$ under $H^i_\frakm(R')\to H^i_\frakm(S)$. 

By Theorem~\ref{thm:HL}, there exists a module-finite extension $R_1$ of $R'$ such that $[\eta]$ maps to $0$ under $H^i_{\frakm}(R')\to H^i_{\frakm}(R_1)$. Setting $R_2$ to be the separable closure of $R'$ in $R_1$, the image of $[\eta]$ in $H^i_{\frakm}(R_2)$ lies in $H^i_{\frakm}(R_2)_{nil}$. The result now follows by (1).
\end{proof}

\begin{corollary}
\label{corollary:sepCM}
Let $(R,\frakm)$ be a local domain of prime characteristic that is a homomorphic image of a Gorenstein ring. Then $H^i_\frakm(R^\sep)=0$ for each $i<\dim R$.

Moreover, each system of parameters for $R$ is a regular sequence on $R^\sep$, i.e., $R^\sep$ is a separable balanced big Cohen-Macaulay algebra for $R$.
\end{corollary}

\begin{proof}
Theorem~\ref{thm:main:cohomology}\,(2) implies that $H^i_\frakm(R^\sep)=0$ for each $i<\dim R$. The proof that this implies the second statement is similar to the proof of \cite[Corollary~2.3]{HL}.
\end{proof}

\begin{proof}[Proof of Theorem~\ref{thm:main:ideal}]
Let $p$ be the characteristic of $R$. If $z\in\fraka^F$, there exists a prime power $q=p^e$ with $z^q\in\fraka^{[q]}$. In this case, $z^{q/p}$ belongs to the Frobenius closure of $\fraka^{[q/p]}$, and
\[
(z^{q/p})^p\in(\fraka^{[q/p]})^{[p]}\,.
\]
Since solvable extensions form a distinguished class, we reduce to the case $e=1$, i.e., $q=p$.

There exist nonzero elements, $a_0,\dots,a_m\in\fraka$ and $b_0,\dots,b_m\in R$ with
\[
z^p=\sum_{i=0}^mb_ia_i^p\,.
\]
Consider the polynomials
\[
T^p+a_0^pT-b_i\qquad\text{ for }i=1,\dots,m\,,
\]
and let $L$ be a finite extension field where these have roots $t_1,\dots,t_m$ respectively. By Lemma~\ref{lemma:solvable}, we may assume $L$ is Galois over $\ff(R)$ with the Galois group being solvable. Set
\begin{equation}\label{eqn:key}
t_0=\frac{1}{a_0}\Big(z-\sum_{i=1}^mt_ia_i\Big)\,.
\end{equation}
Taking $p$-th powers, we have
\[
t_0^p\ =\ \frac{1}{a_0^p}\Big(\sum_{i=0}^mb_ia_i^p-\sum_{i=1}^mt_i^pa_i^p\Big)\
=\ b_0 + \frac{1}{a_0^p}\sum_{i=1}^m\big(b_i-t_i^p\big)a_i^p
=\ b_0 + \sum_{i=1}^mt_ia_i^p\,.
\]
Thus, $t_0$ belongs to the integral closure of $R[t_1,\dots,t_m]$ in its field of fractions. Let $S$ be a module-finite extension of $R$ that contains $t_0,\dots,t_m$, and has $L$ as its fraction field; if $R$ is excellent, we may take $S$ to be the integral closure of $R$ in $L$. Since~\eqref{eqn:key} may be rewritten as
\[
z=\sum_{i=0}^mt_ia_i\,,
\]
it follows that $z\in\fraka S$, completing the proof of (1).

(2) follows from \cite[Corollary~3.4]{Singh:sep}, though we include a proof using~(1). There exists a module-finite extension domain $T$ such that $z\in\fraka T$. Decompose the field extension $\ff(R)\subseteq\ff(T)$ as a separable extension $\ff(R)\subseteq\ff(T)^{\sep}$ followed by a purely inseparable extension $\ff(T)^{\sep}\subseteq\ff(T)$. Let $T_0$ be the integral closure of $R$ in $\ff(T)^{\sep}$.

Since $T$ is a purely inseparable extension of $T_0$, and $z\in\fraka T$, it follows that 
$z$ belongs to the Frobenius closure of the ideal $\fraka T_0$. By (2) there exists a generically separable extension $S_0$ of $T_0$ with $z\in\fraka S_0$. Enlarge $S_0$ to a generically Galois extension $S$ of $R$. This concludes the argument in the case $R$ is excellent; in the event that $S$ is not module-finite over $R$, one may replace it by a subring satisfying $z\in\fraka S$ and having the same fraction field.
\end{proof}

The equational construction used in the proof of Theorem~\ref{thm:main:ideal}\,(1) arose from the study of symplectic invariants in \cite{Singh:inv}.

\section{Some Galois groups that are not solvable}

Let $R$ be a domain of prime characteristic, and let $\fraka$ be an ideal of $R$. If $z$ is an element of $\fraka^F$, Theorem~\ref{thm:main:ideal}\,(1) states that there exists a solvable module-finite extension $S$ with $z\in\fraka S$. In the following example one has $z\in\fraka^+$, and we conjecture $z\notin\fraka S$ for any module-finite generically Galois extension $S$ with $\Gal(S/R)$ solvable.

\begin{example}
\label{example:plus}
Let $a,b,c_1,c_2$ be algebraically independent over $\FF_p$, and set $R$ be the hypersurface
\[
\frac{\FF_p(a,b,c_1,c_2)[x,y,z]}{\big(z^{p^2}+c_1(xy)^{p^2-p}z^p+c_2(xy)^{p^2-1}z+ax^{p^2}+by^{p^2}\big)}\,.
\]
We claim $z\in(x,y)^+$. Let $u,v$ be elements of $R^+$ that are, respectively, roots of the polynomials
\begin{equation}
\label{eq:dickson:eq1}
T^{p^2}+c_1y^{p^2-p}T^p+c_2y^{p^2-1}T+a\,,
\end{equation}
and
\[
T^{p^2}+c_1x^{p^2-p}T^p+c_2x^{p^2-1}T+b\,.
\]
Set $S$ to be the integral closure of $R$ in the Galois closure of $\ff(R)(u,v)$ over $\ff(R)$. Then $(z-ux-vy)/xy$ is an element of $S$, since it is a root of the monic polynomial
\[
T^{p^2}+c_1T^p+c_2T\,.
\]
It follows that $z\in(x,y)S$.

We next show that $\Gal(S/R)$ is not solvable for the extension $S$ constructed above. Since $u$ is a root of~\eqref{eq:dickson:eq1}, $u/y$ is a root of
\begin{equation}
\label{eq:dickson:eq2}
T^{p^2}+c_1T^p+c_2T+\frac{a}{y^{p^2}}\,.
\end{equation}
The polynomial~\eqref{eq:dickson:eq2} is irreducible over $\FF_q(c_1,c_2,a/y^{p^2})$, and hence over the purely transcendental extension $\FF_q(c_1,c_2,a,x,y,z)=\ff(R)$. Since $\ff(S)$ is a Galois extension of $\ff(R)$ containing a root of~\eqref{eq:dickson:eq2}, it contains all roots of~\eqref{eq:dickson:eq2}. As~\eqref{eq:dickson:eq2} is separable, its roots are distinct; taking differences of roots, it follows that $\ff(S)$ contains the $p^2$ distinct roots of
\begin{equation}
\label{eq:dickson:eq3}
T^{p^2}+c_1T^p+c_2T\,.
\end{equation}
We next verify that the Galois group of~\eqref{eq:dickson:eq3} over $\ff(R)$ is $\GL_2(\FF_q)$.

Quite generally, let $L$ be a field of characteristic $p$. Consider the standard linear action of $\GL_2(\FF_p)$ on the polynomial ring $L[x_1,x_2]$. The ring of invariants for this action is generated over $L$ by the \emph{Dickson invariants} $c_1,c_2$, which occur as the coefficients in the polynomial
\[
\prod_{\alpha,\beta\in\FF_p}(T-\alpha x_1-\beta x_2)=T^{p^2}+c_1T^p+c_2T\,,
\]
see \cite{Dickson} or \cite[Chapter~8]{Benson}. Hence the extension $L(x_1,x_2)/L(c_1,c_2)$ has Galois group $\GL_2(\FF_p)$.

It follows from the above that if $c_1,c_2$ are algebraically independent elements over a field $L$ of characteristic $p$, then the polynomial
\[
T^{p^2}+c_1T^p+c_2T\ \in\ L(c_1,c_2)[T]
\]
has Galois group $\GL_2(\FF_p)$.

The group $\PSL_2(\FF_p)$ is a subquotient of $\GL_2(\FF_p)$, and, we conjecture, a subquotient of $\Gal(S/R)$ for \emph{any} module-finite generically Galois extension $S$ of $R$ with $z\in\fraka S$. For $p\ge 5$, the group $\PSL_2(\FF_p)$ is a nonabelian simple group; thus, conjecturally, $\Gal(S/R)$ is not solvable for any module-finite generically Galois extension~$S$ with $z\in\fraka S$.
\end{example}

\begin{example}
\label{example:plus2}
Extending the previous example, let $a,b,c_1,\dots,c_n$ be algebraically independent elements over $\FF_q$, and set $R$ to be the polynomial ring $\FF_q(a,b,c_1,\dots,c_n)[x,y,z]$ modulo the principal ideal generated by
\begin{multline*}
z^{q^n}+c_1(xy)^{q^n-q^{n-1}}z^{q^{n-1}}+c_2(xy)^{q^n-q^{n-2}}z^{q^{n-2}}+\cdots+c_n(xy)^{q^n-1}z\\
+ax^{q^n}+by^{q^n}\,.
\end{multline*}
Then $z\in(x,y)^+$; imitate the previous example with $u,v$ being roots of 
\[
T^{q^n}+c_1y^{q^n-q^{n-1}}T^{q^{n-1}}+c_2y^{q^n-q^{n-2}}T^{q^{n-2}}+\cdots+c_ny^{q^n-1}T+a\,,
\]
and
\[
T^{q^n}+c_1x^{q^n-q^{n-1}}T^{q^{n-1}}+c_2x^{q^n-q^{n-2}}T^{q^{n-2}}+\cdots+c_nx^{q^n-1}T
+b\,.
\]

If $S$ is any module-finite generically Galois extension of $R$ with $z\in\fraka S$, we conjecture that $\ff(S)$ contains the splitting field of 
\begin{equation}
\label{eqn:dickson:n}
T^{q^n}+c_1T^{q^{n-1}}+c_2T^{q^{n-2}}+\cdots+c_nT
\,.
\end{equation}
Using a similar argument with Dickson invariants, the Galois group of~\eqref{eqn:dickson:n} over $\ff(R)$ is $\GL_n(\FF_q)$. Its subquotient $\PSL_n(\FF_q)$ is a nonabelian simple group for $n\ge 3$, and for $n=2$, $q\ge 4$.
\end{example}

Likewise, we record conjectural examples $R$ where $H^i_\frakm(R)\to H^i_\frakm(S)$ is nonzero for each module-finite generically Galois extension $S$ with $\Gal(S/R)$ solvable:

\begin{example}
\label{example:lc}
Let $a,b,c_1,c_2$ be algebraically independent over $\FF_p$, and consider the hypersurface
\[
A=\frac{\FF_p(a,b,c_1,c_2)[x,y,z]}{\big(z^{2p^2}+c_1(xy)^{p^2-p}z^{2p}+c_2(xy)^{p^2-1}z^2+ax^{p^2}+by^{p^2}\big)}\,.
\]
Let $(R,\frakm)$ be the Rees ring $A[xt,yt,zt]$ localized at the maximal ideal $x,y,z,xt,yt,zt$. The elements $x,yt,y+xt$ form a system of parameters for $R$, and the relation
\[
z^2t\cdot(y+xt)\ =\ z^2t^2\cdot x+ z^2\cdot yt
\]
defines an element $[\eta]$ of $H^2_\frakm(R)$. We conjecture that if $S$ is any module-finite generically Galois extension such that $[\eta]$ maps to $0$ under the induced map $H^2_\frakm(R)\to H^2_\frakm(S)$, then $\ff(S)$ contains the splitting field of 
\[
T^{p^2}+c_1T^p+c_2T\,,
\]
and hence that $\Gal(S/R)$ is not solvable if $p\ge5$.
\end{example}

\section{Graded rings and extensions}

Let $R$ be an $\NN$-graded domain that is finitely generated over a field $R_0$. Set $R^\GR$ to be the $\QQ_{\ge0}$-graded ring generated by elements of $R^+$ that can be assigned a degree such that they then satisfy a homogeneous equation of integral dependence over $R$. Note that ${[R^\GR]}_0$ is the algebraic closure of the field $R_0$. One has the following:

\begin{theorem}\cite[Theorem~6.1]{HHbig}
\label{thm:hh:graded}
Let $R$ be an $\NN$-graded domain that is finitely generated over a field $R_0$ of prime characteristic. Then each homogeneous system of parameters for $R$ is a regular sequence on $R^\GR$.
\end{theorem}

Let $R$ be as in the above theorem. Since $R^\GR$ and $R^\sep$ are Cohen-Macaulay $R$-algebras, it is natural to ask whether there exists a $\QQ$-graded separable Cohen-Macaulay $R$-algebra. The answer to this is negative:

\begin{example}
\label{example:grsep}
Let $R$ be the Rees ring
\[
\frac{\bar{\FF}_2[x,y,z]}{(x^3+y^3+z^3)}[xt,yt,zt]
\]
with the $\NN$-grading where the generators $x,y,z,xt,yt,zt$ have degree~$1$. Set $B$ to be the $R$-algebra generated by the homogeneous elements of $R^\GR$ that are separable over $\ff(R)$. We prove that~$B$ is not a balanced Cohen-Macaulay $R$-module.

The elements $x$, $yt$, $y+xt$ are a system of parameters for $R$. Suppose, to the contrary, that they form a regular sequence on $B$. Since
\[
z^2t\cdot(y+xt)\ =\ z^2t^2\cdot x+ z^2\cdot yt\,,
\]
it follows that $z^2t\in(x,yt)B$. Thus, there exist elements $u,v\in B_1$ with
\begin{equation}\label{eqn:grsep:1}
z^2t\ =\ u\cdot x\ +\ v\cdot yt\,.
\end{equation}
Since $z^3=x^3+y^3$, we also have $z^2\ =\ x\sqrt{xz}+y\sqrt{yz}$ in $R^\GR$, and hence
\begin{equation}\label{eqn:grsep:2}
z^2t\ =\ t\sqrt{xz}\cdot x\ +\ \sqrt{yz}\cdot yt\,.
\end{equation}
Comparing~\eqref{eqn:grsep:1} and~\eqref{eqn:grsep:2}, we see that
\[
(u+t\sqrt{xz})\cdot x\ =\ (v+\sqrt{yz})\cdot yt
\]
in $R^\GR$. But $x,yt$ is a regular sequence on $R^\GR$, so there exists an element $c$ in ${[R^\GR]}_0$ with $u+t\sqrt{xz}=cyt$ and $v+\sqrt{yz}=cx$. Since ${[R^\GR]}_0=\bar{\FF}_2$, it follows that $c\in R$, and hence that $\sqrt{yz}\in B$. This contradicts the hypothesis that elements of $B$ are separable over $\ff(R)$.

The above argument shows that any graded Cohen-Macaulay $R$-algebra must contain the elements 
$\sqrt{yz}$ and $t\sqrt{xz}$.
\end{example}

We next show that no module-finite $\QQ$-graded extension domain of the ring $R$ in Example~\ref{example:grsep} is Cohen-Macaulay. 

\begin{example}
\label{example:nfgcm}
Let $R$ be the Rees ring from Example~\ref{example:grsep}, and let $S$ be a graded Cohen-Macaulay ring with $R\subseteq S\subseteq R^\GR$. We prove that $S$ is not finitely generated over $R$.

By the previous example, $S$ contains $\sqrt{yz}$ and $t\sqrt{xz}$. Using the symmetry between $x,y,z$, it follows that $\sqrt{xy}$, $\sqrt{xz}$, $t\sqrt{xy}$, $t\sqrt{yz}$ are all elements of $S$. We prove inductively that $S$ contains
\begin{alignat}3\label{eqn:ex1:4}
x^{1-2/q}(yz)^{1/q}, &\qquad& y^{1-2/q}(xz)^{1/q}, &\qquad& z^{1-2/q}(xy)^{1/q},\\
tx^{1-2/q}(yz)^{1/q}, &\qquad& ty^{1-2/q}(xz)^{1/q}, &\qquad& tz^{1-2/q}(xy)^{1/q},\nonumber
\end{alignat}
for each $q=2^e$ with $e\ge 1$. The case $e=1$ has been settled.

Suppose $S$ contains the elements \eqref{eqn:ex1:4} for some $q=2^e$. Then, one has
\begin{multline*}
x^{1-2/q}(yz)^{1/q}\cdot ty^{1-2/q}(xz)^{1/q}\cdot(y+xt)\\
=\ tx^{1-2/q}(yz)^{1/q}\cdot ty^{1-2/q}(xz)^{1/q}\cdot x\ +\ x^{1-2/q}(yz)^{1/q}\cdot y^{1-2/q}(xz)^{1/q}\cdot yt\,.
\end{multline*}
Using as before that $x$, $yt$, $y+xt$ is a regular sequence on $S$, we conclude
\[
x^{1-2/q}(yz)^{1/q}\cdot ty^{1-2/q}(xz)^{1/q}\ =\ u\cdot x\ +\ v\cdot yt
\]
for some $u,v\in S_1$. Simplifying the left hand side, the above reads
\begin{equation}\label{eqn:ex1:5}
t(xy)^{1-1/q}z^{2/q}\ =\ u\cdot x\ +\ v\cdot yt\,.
\end{equation}
Taking $q$-th roots in
\[
z^2\ =\ x\sqrt{xz}+y\sqrt{yz}
\]
and multiplying by $t(xy)^{1-1/q}$ yields
\begin{equation}\label{eqn:ex1:6}
t(xy)^{1-1/q}z^{2/q}\ =\ ty^{1-1/q}(xz)^{1/2q}\cdot x\ +\ x^{1-1/q}(yz)^{1/2q}\cdot yt\,.
\end{equation}
Comparing~\eqref{eqn:ex1:5} and~\eqref{eqn:ex1:6}, we see that
\[
\big(u+ty^{1-1/q}(xz)^{1/2q}\big)\cdot x\ =\ \big(v+x^{1-1/q}(yz)^{1/2q}\big)\cdot yt\,,
\]
so there exists $c$ in ${[R^\GR]}_0=\bar{\FF}_2$ with
\[
u+ty^{1-1/q}(xz)^{1/2q}=cyt\qquad\text{ and }\qquad v+x^{1-1/q}(yz)^{1/2q}=cx\,.
\]
It follows that $ty^{1-1/q}(xz)^{1/2q}$ and $x^{1-1/q}(yz)^{1/2q}$ are elements of $S$. In view of the symmetry between $x,y,z$, this completes the inductive step. Setting
\[
\theta=\frac{xy}{z^2}\,,
\]
we have proved that
\[
\theta^{1/q}\ \in\ \ff(S)\qquad\text{ for each }q=2^e\,.
\]

We claim $\theta^{1/2}$ does not belong to $\ff(R)$. Indeed if it does, then $(xy)^{1/2}$ belongs to $\ff(R)$, and hence to $R$, as $R$ is normal; this is readily seen to be false.
The extension
\[
\ff(R)\ \subseteq\ \ff(R)(\theta^{1/q})
\]
is purely inseparable, so the minimal polynomial of $\theta^{1/q}$ over $\ff(R)$ has the form $T^Q-\theta^{Q/q}$ for some $Q=2^E$. Since $\theta^{1/2}\notin\ff(R)$, we conclude that the minimal polynomial is $T^q-\theta$. Hence
\[
\left[\ff(R)(\theta^{1/q}):\ff(R)\right]=q\qquad\text{ for each }q=2^e\,.
\]
It follows that $[\ff(S):\ff(R)]$ is not finite.
\end{example}

Theorem~\ref{thm:HL} and Theorem~\ref{thm:main:cohomology}\,(2) discuss the vanishing of the image of $H^i_\frakm(R)$ for $i<\dim R$. In the case of graded rings, one also has the following result for $H^d_\frakm(R)$.

\begin{proposition}
\label{prop:top}
Let $R$ be an $\NN$-graded domain that is finitely generated over a field $R_0$ of prime characteristic. Set $d=\dim R$. Then the submodule ${[H^d_\frakm(R)]}_{\ge 0}$ maps to zero under the induced map
\[
H^d_\frakm(R)\to H^d_\frakm(R^\GR)\,.
\]
Hence, there exists a module-finite $\QQ$-graded extension domain $S$ of $R$ such that the induced map ${[H^d_\frakm(R)]}_{\ge 0}\to H^d_\frakm(S)$ is zero.
\end{proposition}

\begin{proof}
Let $F^e\colon H^d_\frakm(R)\to H^d_\frakm(R)$ denote the $e$-th iteration of the Frobenius map. Suppose $[\eta]\in{[H^d_\frakm(R)]}_n$ for some $n\ge0$. Then $F^e([\eta])$ belongs to ${[H^d_\frakm(R)]}_{np^e}$ for each $e$. As ${[H^d_\frakm(R)]}_{\ge 0}$ has finite length, there exists $e\ge1$ and homogeneous elements $r_1,\dots,r_e\in R$ such that
\begin{equation}
\label{eqn:iterates}
F^e([\eta])+r_1F^{e-1}([\eta])+\cdots+r_e[\eta]=0\,.
\end{equation}
We imitate the equational construction from \cite{HL}: Consider a homogeneous system of parameters $\bsx=x_1,\dots,x_d$, and compute $H^i_\frakm(R)$ as the cohomology of the \v Cech complex $C^\bullet(\bsx;R)$ below:
\[
0\to R\to\bigoplus_{i=1}^dR_{x_i}\to\bigoplus_{i<j}R_{x_ix_j}\to\cdots\to R_{x_1\cdots x_d}\to 0\,.
\]
This complex is $\ZZ$-graded; let $\eta$ be a homogeneous element of $C^d(\bsx;R)$ that maps to $[\eta]$ in $H^d_\frakm(R)$. Equation~\eqref{eqn:iterates} implies that
\[
F^e(\eta)+r_1F^{e-1}(\eta)+\cdots+r_e\eta
\]
is a boundary in $C^d(\bsx;R)$, say it equals $\del(\alpha)$ for a homogeneous element $\alpha$ of~$C^{d-1}(\bsx;R)$. Solving integral equations in each coordinate of $C^{d-1}(\bsx;R)$, there exists a module-finite extension domain~$S$ and $\beta$ in $C^{d-1}(\bsx;S)$ with
\[
F^e(\beta)+r_1F^{e-1}(\beta)+\cdots+r_e\beta=\alpha\,.
\]
Moreover, we may assume $S$ is a normal ring. Since $\eta-\del(\beta)$ is an element on $\ff(S)$ satisfying
\[
T^{p^e}+r_1T^{p^{e-1}}+\cdots+r_eT=0\,,
\]
it belongs to $S$. But then $\eta-\del(\beta)$ maps to zero in $H^d_\frakm(S)$. Thus, each homogeneous element of ${[H^d_\frakm(R)]}_{\ge 0}$ maps to $0$ in $H^d_\frakm(R^\GR)$.

For the final statement, note that ${[H^d_\frakm(R)]}_{\ge 0}$ has finite length.
\end{proof}

The next example illustrates why Proposition~\ref{prop:top} is limited to ${[H^d_\frakm(R)]}_{\ge 0}$.

\begin{example}
Let $K$ be a field of prime characteristic, and take $R$ to be the semigroup ring
\[
R=K[x_1\cdots x_d,\ x_1^d,\ \dots,\ x_d^d]\,.
\]
It is easily seen that $R$ is normal, and that ${[H^d_\frakm(R)]}_n$ is nonzero for each integer $n<0$. We claim that the induced map
\[
H^d_\frakm(R)\to H^d_\frakm(S)
\]
is injective for each module-finite extension ring $S$. For this, it suffices to check that $R$ is a \emph{splinter} ring, i.e., that $R$ is a direct summand of each module-finite extension ring; the splitting of $R\subseteq S$ then induces an $R$-splitting of $H^d_\frakm(R)\to H^d_\frakm(S)$.

To check that $R$ is a splinter ring, note that normal affine semigroup rings are weakly $F$-regular by \cite[Proposition~4.12]{HHJAMS}, and that weakly $F$-regular rings are splinter by \cite[Theorem~5.25]{HHJALG}. For more on splinters, we point the reader towards \cite{Ma,HHJALG,Singh:splinter1}.
\end{example}

\subsection*{Acknowledgments}
We thank Kazuhiko Kurano for pointing out an error in an earlier version of this manuscript.


\end{document}